\documentclass{amsart}

\usepackage{amssymb}
\usepackage{epsfig} 
\usepackage[mathscr]{euscript}
\usepackage{mathtools}
\usepackage{amsmath}
\usepackage{tikz}
\usetikzlibrary{calc,matrix}
\usepackage{xcolor}
\usepackage{setspace}
\usepackage{blindtext}

\usepackage{filecontents}

\newtheorem{thm}{Theorem}[section]
\newtheorem{prop}[thm]{Proposition}
\newtheorem{lem}[thm]{Lemma}
\newtheorem{cor}[thm]{Corollary}

\theoremstyle{definition}
\newtheorem{definition}[thm]{Definition}
\newtheorem{example}[thm]{Example}

\title{}
\date{}
\author{}

\newcommand{\ais}[3]{{#3}_{#1}^{\scriptscriptstyle \leqslant {#2}}}
\newcommand{\cais}[3]{{#3}_{#1}^{\scriptscriptstyle \geqslant {#2}}}
\newcommand{\ctr}[2]{{\tau}_{#1}^{\scriptscriptstyle \leqslant {#2}}}
\newcommand{\cotr}[2]{{\tau}_{#1}^{\scriptscriptstyle \geqslant {#2}}}

\newcommand{\coli}{\text{colim}}

\newcommand{\ho}{\text{Hom}}
\newcommand{\Db}{\text{D}^{ b}}
\newcommand{\Perf}{\text{Perf}}
\newcommand{\D}{\text{D}}
\newcommand{\Spec}{\text{Spec}R}
\newcommand{\overbar}[1]{\mkern 1.5mu\overline{\mkern-1.5mu#1\mkern-1.5mu}\mkern 1.5mu}
\newcommand{\K}{\text{K}}
\newcommand{\U}{\mathscr{U}}
\newcommand{\V}{\mathscr{V}}
\newcommand{\C}{\mathscr{C}}

\begin{document}

\title[t-structures on Perfect complexes]{Bounded t-structures on the category of Perfect complexes over a Noetherian ring of finite Krull dimension}

\author{Harry Smith}

\date{October \\ 2019}

\maketitle

\begin{abstract}
We classify bounded t-structures on the category of perfect complexes over a commutative, Noetherian ring of finite Krull dimension, 
extending a result of Alonso Tarrio, Jeremias Lopez and Saorin which 
covers the regular case. In particular,
we show that there are no bounded t-structures in the singular case, verifying the affine version of a conjecture of Antieau, Gepner and Heller, and also that 
there are no non-trivial t-structures at all in the singular, irreducible case.
\end{abstract}

\tableofcontents

\section{Introduction}
Beilinson, Bernstein and Deligne introduced t-structures in order to construct the category of perverse sheaves over an algebraic or analytic variety \cite{bbd}. 
A t-structure on a triangulated category (or a stable
$\infty$-category) consists of two full 
subcategories, called the aisle and coaisle, satisfying axioms which abstract the relationship between the connective and coconnective parts of a category of complexes. 
It follows from the axioms that each t-structure comes equipped with a natural
 cohomological functor from the original triangulated category to some abelian subcategory, called the heart.

There are various results using information of a geometric nature to characterise t-structures on derived categories. 
In \cite{neem1}, Neeman shows that the smashing localizing subcategories of the unbounded derived category of a commutative ring correspond to specialization closed subsets of that
ring's Zariski spectrum. In \cite{kash1} Kashiwara uses a certain filtration of specialization closed subsets, 
to construct a t-structure on the bounded derived category of coherent sheaves on a complex manifold. In \cite{atjls}, Alonso Tarrio, Jeremias Lopez, and Saorin give a 
full classification for compactly generated 
t-structures on the unbounded derived category of a Noetherian ring, showing that they are in bijection with filtrations of specialization closed subsets. 
Furthermore, for a large class of Noetherian rings, they identify those t-structures restricting to the bounded derived category with filtrations satisfying a condition called the 
weak cousin condition. In \cite{hrbek}, Hrbek extends the classification on the unbounded derived category to arbitrary commutative rings by replacing specialization closed subsets with
Thomason subsets. These classifications are discussed in sections 3 and 4.

The t-structures referred to above are compactly generated, meaning that they can each be characterised by a collection of compact objects.
The compact objects in the derived category of a commutative ring are exactly the perfect complexes \cite{stacks1}, that is the complexes quasi-isomorphic to a bounded complex of finitely generated projectives. 
These form a triangulated subcategory, denoted $\Perf(R)$. 

To a commutative ring we may assign its algebraic $K$-theory spectrum $\K(R)$, from which its algebraic $K$-groups are computed. This spectrum may be realized as the image of $\Perf(R)$,
under the $K$-theory functor for small, idempotent complete, stable $\infty$-categories $\K: \text{Cat}_{\infty}^{\text{perf}} \to \text{Sp}$, see \cite{bgt1}. Barwick's theorem of the heart \cite{clarwick1} connects 
the connective $K$-theory of a small, stable $\infty$-category with the existence of bounded t-structures on that category. Furthermore in \cite{agh1}, Antieau, Gepner and Heller
show that when the heart of this t-structure is Noetherian, the theorem of the heart also holds for connective $K$-theory. Specifically they prove the following theorem.

\begin{thm}
\label{aghtheorem}
Let $\mathscr{C}$ be a small, stable $\infty$-category with a bounded t-structure.  Then $\emph{K}_{-1}(\mathscr{C}) = 0$ and, if $\mathscr{C}^\heartsuit$ is also Noetherian,
then $\emph{K}_{-n}(\mathscr{C}) = 0$ for $n \geqslant 1$.
\end{thm}

As a consequence of this, we see that non-trivial negative $K$-groups can give an obstruction to the existence of bounded t-structures on $\Perf(R)$. 
We know from \cite[2.7.8]{weib2} that negative $K$-groups vanish over a commutative, regular, Noetherian ring, so this observation only applies in the singular case. 
We prove the following theorem, which is the affine case of a conjecture by Antieau, Gepner and Heller \cite{agh1}.

\begin{thm}
Let $R$ be a singular, Noetherian ring of finite Krull dimension. Then $\emph{Perf}(R)$ admits no bounded t-structure.
\end{thm}

This yields a classification for bounded t-structures on $\Perf(R)$ for $R$ commutative, Noetherian and with finite Krull dimension.
The regular case follows from results in \cite{atjls} (see proposition \ref{boundedness}). Furthermore, we prove the following result about arbitrary t-structures.

\begin{thm}
Let $R$ be a singular, irreducible, Noetherian ring of finite Krull dimension $d$, and let $(\U, \V)$ be a t-structure on $\emph{Perf}(R)$. 
Then either $\U = 0$ or $\U = \emph{Perf}(R)$.
\end{thm}

The arguments used to show the above rely on the observation that t-structures on $\Perf(R)$ will  extend to $\D(R)$. This allows us to invoke a classification 
result due to Alonso Tarrio, Jeremias Lopez, and Saorin \cite{atjls}, identifying compactly generated t-structures on $\D(R)$ for Noetherian $R$
with decreasing filtrations of Thomason subsets of $\Spec$.

Asking if a given t-structure restricts to a triangulated subcategory is equivalent to asking if the corresponding truncation functors preserve that subcategory.
This means that in order to prove the above two theorems, it will suffice to show that for each t-structure we can choose a perfect complex whose truncation
is not perfect. In the singular case, there are two main types of obstruction preventing the various t-structures on $\D(R)$ from restricting to  $\Perf(R)$.

Firstly, for t-structures corresponding to filtrations terminating below at the empty set, it is possible to show that a truncation of the Koszul complex corresponding to 
the singular point will give the residue field of that point. In this case, this residue field has infinite projective dimension, and could not be a perfect complex
(see  lemma \ref{doesnotrestrict}).
In the remaining cases, we observed that the cohomology of the truncation of a module would contain local cohomology groups, up to localization at an appropriate prime 
(see proposition \ref{summand}).  Since local cohomology modules are often infinitely generated over $R$, this also shows that the truncation cannot be a perfect complex.

In \cite{bridge1} Bridgeland introduced stability conditions for triangulated categories. He proved that defining a stability condition on a triangulated category is equivalent
to finding a bounded t-structure, and then defining a stability function on its heart, satisfying the Harder-Narasimham property. 
As such our result also proves the non-existence of stability conditions on $\Perf(R)$ for singular $R$.

\subsection{Notation and conventions}

Let $R$ denote a commutative ring with identity throughout.  Let $\text{Mod}(R)$ denote the abelian category of $R$-modules. 
Let $\text{D}(R)$ denote the \textbf{unbounded derived category} of $R$, that is the triangulated category of chain complexes of $R$-modules up 
to quasi-isomorphism. For $R$ Noetherian, let $\Db(R)$ denote the \textbf{bounded derived category} of $R$, that is the triangulated subcategory of $\text{D}(R)$ of complexes with finitely generated and bounded 
homology, up to quasi-isomorphism. Let $\text{Perf}(R) \subseteq \Db(R)$ denote the triangulated category of \textbf{perfect complexes} over $R$, that is the complexes quasi-isomorphic to a bounded complex of 
finitely generated projective 
modules. 

For any category $\mathscr{C}$ and any collection of objects $\U \subseteq \mathscr{C}$, define two associated collections of objects, the right-orthogonal and left-orthogonal
respectively, given by
\[ \mathscr{U}^{\perp} = \{ Y \in \mathscr{C}\ | \ \ho_{\mathscr{C}}(X,Y) \simeq 0 \ \text{all} \ X \in 
\mathscr{U} \} \]
\[ ^{\perp}\!\mathscr{U} = \{ X \in \mathscr{C}\ | \ \ho_{\mathscr{C}}(X,Y) \simeq 0 \ \text{all} \ Y \in 
\mathscr{U} \}. \]

We use cohomological indexing conventions. On a triangulated category (or stable $\infty$-category) $\mathscr{C}$, let $[1]$ denote the suspension functor, and let $[n]$ denote its $n$th iteration.
This convention is such that, for $M, N \in \text{Mod}(R)$, viewed as objects of $\D(R)$, we have $\text{Ext}_R^n(M, N) = \text{Hom}_{\D(R)}(M, N[n])$.
A t-structure on a stable $\infty$-category is by definition a t-structure on its homotopy category, see \cite[1.2.1.4]{lur3}. Therefore, for the rest of this paper we will only use the language of 
triangulated categories.

\vspace{0.1cm}

\section{Preliminaries on t-structures}

We recall basic results on t-structures. 
\begin{definition} 
\label{tstr}
A \textbf{t-structure} on a triangulated category $\mathscr{C}$, consists of a pair of full 
subcategories $(\mathscr{U}, \mathscr{V})$  satisfying the following
conditions:
\begin{enumerate}
\item $\mathscr{U}[1] \subseteq \mathscr{U}$ and 
$\mathscr{V} \subseteq \mathscr{V}[1]$.

\item $\text{Hom}_{\mathscr{C}}(X[1], Y)=0$ for any $X \in \mathscr{U}$
and $Y \in \mathscr{V}$.

\item For any $X \in \mathscr{C}$ we have a triangle $\ctr{}{0}X \to
X \to \cotr{}{1}X$ where $\ctr{}{0}X \in \mathscr{U}$ and
$\cotr{}{1}X \in \mathscr{V}[-1]$.
\end{enumerate}
\end{definition}
We call $\U$ the \textbf{aisle} and $\V$
the \textbf{coaisle}. We set $\ais{}{n}{\U} = \U[-n]$ and 
$\cais{}{n}{\V} = \V[-n]$, but we will drop the superscript for $n=0$. The category $\mathscr{C}^{\heartsuit} = \U \cap \V$ 
is called the \textbf{heart} of the t-structure. For triangulated categories $\C$ and $\mathscr{D}$, with t-structures 
$(\U, \V)$ and $(\mathscr{E}, \mathscr{F})$ respectively, an exact functor $\mathscr{F}: 
\mathscr{C} \to \mathscr{D}$  is said to be \textbf{t-exact} if 
$\mathscr{F}(\U) \subseteq \mathscr{E}$
and $\mathscr{F}(\V) \subseteq \mathscr{F}$. 
We have the following well known facts about t-structures, 
see \cite[1.2.1]{lur3}.

\begin{enumerate}
\item Both $\ais{}{n}{\U}$ and $\cais{}{n}{\V}$ are closed under extensions.

\item We have $\ais{}{n}{\U} = {}^{\perp}\!(\cais{}{n+1}{\V})$ and
${(\ais{}{n}{\U}})^{\perp} = \cais{}{n+1}{\V}$, and so any t-structure is fully determined by either its aisle or coaisle.

\item The inclusion $\ais{}{n}{\U} \to \mathscr{C}$ (resp. 
$\cais{}{n}{\V} \to \mathscr{C}$) has a right (resp. left) adjoint,
denoted $\ctr{}{n}$ (resp. $\cotr{}{n}$), and called the 
\textbf{connective (resp. coconnective) truncation functors} of the t-structure.

\item The functors $\ctr{}{0}$ and $\cotr{}{1}$ coincide with the triangle in part (3)
of definition \ref{tstr}, and this triangle is necessarily unique.

\item $\mathscr{C}^{\heartsuit}$ is an abelian subcategory of $\mathscr{C}$ and the functor \[\ctr{}{0} \circ \cotr{}{0}: \mathscr{C} \to \mathscr{C}^{\heartsuit}\]
is cohomological. That is to say, it takes triangles in $\mathscr{C}$ to long exact sequences in $\mathscr{C}^{\heartsuit}$.

\end{enumerate}
A t-structure is said to be generated by a collection of objects $\mathscr{S}$ if
$\V = \mathscr{S}^{\perp}$. If $\mathscr{C}$ is compactly generated, and $\mathscr{S}$ is composed
of compact objects, then the t-structure is said to be \textbf{compactly generated}. A t-structure is said to be \textbf{bounded} if the inclusion
\[ \bigcup_{n \to \infty} \ais{}{n}{\U} \cap \cais{}{-n}{\V} \to \mathscr{C}  \]
is an equivalence. This amounts to asking that for all $X \in \mathscr{C}$, there is $m$ with $X \in \ais{}{m}{\U}$, and there is $n$ with $X \in \cais{}{n}{\V}$. 
It only makes sense to consider the boundedness of t-structures on smaller categories.
In fact, is easy to see that there can be no bounded t-structure on any non-zero triangulated category with arbitrary coproducts.

Let $\mathscr{C}$ be a triangulated category and let $\mathscr{D} \subseteq \mathscr{C}$ be a triangulated subcategory. 
Let $(\mathscr{U}, \mathscr{V})$ be a t-structure on $\mathscr{C}$, and let 
$(\mathscr{E}, \mathscr{F})$ be a t-structure on $\mathscr{D}$. If $\mathscr{E} = \mathscr{U} \cap \mathscr{D} $ and $\mathscr{F} = \mathscr{V} \cap \mathscr{D}$, then we say that $\mathscr{U}$ \textbf{restricts} to $\mathscr{E}$ and $\mathscr{E}$ 
\textbf{extends} to $\mathscr{U}$. 
In this case, it is clear that the trunctation functors on $\mathscr{D}$ are the restrictions of those on $\mathscr{C}$, and that the inclusion functor
$\mathscr{D} \hookrightarrow \mathscr{C}$ is t-exact. 

\begin{prop}
\label{exttstr}
Let $\mathscr{C}$ be a small, triangulated category with a t-structure $(\mathscr{E}, \mathscr{F})$. This t-structure extends 
to a compactly generated t-structure $(\mathscr{U}, \mathscr{V})$ on $\emph{Ind}(\mathscr{C})$, with aisle 
$\mathscr{U} \simeq \emph{Ind}(\mathscr{E})$, and with coaisle 
$\mathscr{V} \simeq \emph{Ind}(\mathscr{F})$.
\end{prop}

\begin{proof}
The fact that this gives a t-structure is \cite[C.2.4.3]{lur1} with the homotopy functor applied. It remains 
to show that this t-structure is compactly generated.
By \cite[5.3.5.5]{lur2}, $\mathscr{C}$ consists of compact objects when viewed as a subcategory 
of $\text{Ind}(\mathscr{C})$. Therefore, to see that $(\U, \V)$ is compactly generated on 
$\text{Ind}(\mathscr{C})$, it suffices to show that $\mathscr{E}{}^{\perp} = 
\V$, with orthogonals taken in $\text{Ind}(\mathscr{C})$.

To see $\V \subseteq \mathscr{E}{}^{\perp}$,
take some $Y \simeq \text{colim}_{j \in J}Y_j \in \V$ with each 
$Y_i \in \mathscr{F}$. Take some $X \in \mathscr{E}$.
Since $(\mathscr{E}, \mathscr{F})$ is a t-structure, we know that 
$\ho_{\mathscr{C}} (X, Y_j) \simeq 0$.
Using the construction of Ind-completions, we see that 
\[ \text{Hom}_{\text{Ind}(\mathscr{C})}(X,Y) \simeq  \coli_j \ho_{\mathscr{C}}
(X, Y_j) \simeq 0. \]
To see $\mathscr{E}{}^{\perp} \subseteq \V $, 
take some $Y \in \mathscr{E}{}^{\perp}$ and take some $X \simeq \text{colim}_{i \in I}X_i \in \U$, such that $X_i \in \mathscr{E}$. We have
\[ \text{Hom}_{\text{Ind}(\mathscr{C})}(X,Y) \simeq  \text{lim}_i \ho_{\mathscr{C}}
(X_i, Y) \simeq 0. \]
Therefore, it follows that $Y \in \U{}^{\perp} = 
\V$ and $(\U, \V)$ is compactly generated by 
$\mathscr{E}$. 
\end{proof}

\vspace{0.1cm}

\section{t-structures on the unbounded derived category}

In this section we discuss the correspondence between t-structures on $\D(R)$ and certain filtrations of subsets of $\Spec$. For some ideal $\mathfrak{a} \subset R$, let 
\[ V(\mathfrak{a}) = \{  \mathfrak{p} \ | \ \mathfrak{p} \supseteq  \mathfrak{a}  \} = \text{Spec}R/\mathfrak{a} \subseteq \Spec \] be the associated Zariski closed set.
Let $M \in \text{Mod}(R)$. Define a subset of $\Spec$ called the \textbf{support} by \[ \text{Supp}M = \{  \mathfrak{p} \ | \ M_{\mathfrak{p}} \neq 0  \}. \] 
Over a Noetherian ring, we have that 
$ \text{Supp}M \subseteq V(\text{ann}M)$, with equality when $M$ is finitely generated \cite[2.2]{bens}. We also have that, for some prime ideal $\mathfrak{p}$,
$M$ is $\mathfrak{p}$-torsion if and only if $\text{Supp}M \subseteq V(\mathfrak{p})$ \cite[2.4]{bens}. 

Let $Z$ be a subset of $\Spec$. We say that $Z$ is \textbf{specialization closed} if for all $\mathfrak{p} \in Z$ and for all $\mathfrak{p} \subseteq \mathfrak{q}$ we have $\mathfrak{q} \in Z$.
Specialization closed subsets are precisely the arbitrary unions of Zariski closed sets in $\Spec$. We further say that $Z$ is a \textbf{Thomason subset} if it is an arbitrary union of Zariski closed sets with quasi-compact
complement. Clearly Thomason subsets are specialization closed.
Over a Noetherian ring every open set of $\Spec$ is quasi-compact, meaning that in this case, the specialization closed subsets are exactly the Thomason subsets. 
 For $Z$, a specialization closed subset or a Thomason subset, define its \textbf{height} by 
\[ \text{height}(Z) \coloneqq \text{inf}\{ \text{height}(\mathfrak{p}) \ | \ 
V(\mathfrak{p}) \subseteq Z \}. \]
 Say that $\mathfrak{p} \in Z$ is a \textbf{minimal prime} of $Z$ if
$\text{height}(\mathfrak{p}) = \text{height}(Z)$. These exist for any non-empty Thomason
subset.

For a finitely generated ideal with a fixed set of generators $(x_1,...,x_r) \subset R$ we define the corresponding \textbf{Koszul complex} 
in $\text{Perf}(R)$ by \[ K(\overbar{x})=\bigotimes_{j=1}^r K(x_j) \] where $K(x_j)=\text{cone}(R[0] \xrightarrow[]{x_j} R[0])$. Note that a different choice of generators for the same ideal does not 
necessarily give a quasi-isomorphic complex, see \cite[1.6.21]{bruns}. Koszul complexes have the following properties:
\begin{enumerate}
\item $\text{H}^0(K(\overbar{x})) = R/(\overbar{x})$

\item The ideal $(\overbar{x})$ annihilates $\text{H}^i(K(\overbar{x}))$ for all $i$, in particular $\text{Supp} \text{H}^i(K(\overbar{x})) \subseteq V(\bar{x})$, \cite[1.6.5(b)]{bruns}.

\item  $\text{H}^i(K(\overbar{x})) = 0$ for all $i \neq 0$, if and only if $(\overbar{x})$ forms a regular sequence on $R$. In this case $K(\overbar{x})$ gives a free resolution of $R/(\overbar{x})$,
 \cite[1.6.14]{bruns}.

\end{enumerate}

\noindent We will now discuss the classifcation for compactly generated t-structures on the unbounded derived category of a ring $R$. 
To some collection of objects $\mathscr{U} \subseteq \D(R)$ such that $\mathscr{U}[1] \subseteq \mathscr{U}$, assign a decreasing filtration $\Phi_{\mathscr{U}}$ of subsets of $\Spec$ by
\[\Phi_{\mathscr{U}} = \left\{ Z^i = \bigcup\{ V(\mathfrak{a}) \ | \ \mathfrak{a} \ \text{f.g. ideal s.t.} \ R/\mathfrak{a}[-i] \in 
	\mathscr{U}	 \}  \right\}_{i \in \mathbb{Z}} \]
Let $\Phi = \{ Z^i \}_{i \in \mathbb{Z}}$ be an decreasing filtration of Thomason 
subsets on $\Spec$.
Define a full subcategory $\mathscr{U}_\Phi \subseteq \D(R)$ and a collection of Koszul complexes $\mathscr{S}_\Phi$ by
\begin{align*}
\mathscr{U}_\Phi    =& \{ X \in \D(R) \ | \ \text{Supp} \text{H}^i(X) \subseteq Z^i, \ i \in \mathbb{Z} \}\\
\mathscr{S}_{\Phi} =& \{ K(\bar{x})[-i] \ |  \ \mathfrak{a} = (\bar{x}) \ \text{f.g. ideal s.t.} \ V(\mathfrak{a}) \subseteq Z^i , \ i \in \mathbb{Z} \}.
\end{align*}
Denote the assignments by $\mu: \mathscr{U} \mapsto \Phi_{\mathscr{U}}$ and $\eta: \Phi \mapsto \mathscr{U}_\Phi$.
We have the following classification result due to Alonso Tarrio, Jeremias Lopez and Saorin \cite[3.11]{atjls}. 
\begin{thm}
\label{classification}
Let $R$ be a Noetherian ring. Then the assignments $\mu$ and $\eta$ give 
a mutually inverse bijection 
\begin{align*}
            \left\{ \parbox[c]{1.8in}{\centering \small
                      \emph{Aisles of compactly generated t-structures on} \emph{D}(R)  }
            \right\}
            \longleftrightarrow 
            \left\{ \parbox[c]{1.4in}{\centering \small
                       \emph{Decreasing Thomason filtrations on} $\emph{Spec}R$}
            \right\} 
\end{align*}
The compact generators of $\mathscr{U}_\Phi$ are given by $\mathscr{S}_\Phi$.
\end{thm}
 In \cite[5.6]{hrbek}, Hrbek used the assignment  $ \Phi \mapsto {}^{\perp}\!({\mathscr{S}_{\Phi}}^{\perp})$ to
show that there is still a bijection in the non-Noetherian case. 
 Below are the most immediate examples of compactly generated t-structures on 
$\text{D}(R)$ with reference to their 
corresponding Thomason filtrations.

\begin{example}[\textbf{Trivial t-structures}]
\label{triv}
We have two trivial t-structures $(\D(R), 0)$ and $(0, \D(R))$, corresponding to the 
constant Thomason filtration at $\Spec$ and the constant Thomason filtration at $\emptyset$ respectively. All triangulated
categories immediately have such t-structures.
\end{example}

\begin{example}[\textbf{Standard t-structure}]
Consider the Thomason filtration $\Phi_{st} = \{ \Spec \ i \leqslant 0, \ \emptyset \ i > 0 \}$. This gives rise to the standard t-structure, with aisle and coaisle given by
\begin{align*}
\mathscr{U}_{st} =& \ \ais{}{0}{\D(R)} =  \{ X \in \text{D}(R) \ | \ \text{H}^i(X) = 0, \ i > 0 \} \\
\mathscr{V}_{st} =& \ \cais{}{0}{\D(R)} = \{ X \in \text{D}(R) \ | \ \text{H}^i(X) = 0, \ i < 0 \}
\end{align*}
and heart $\text{D}(R)_{st}^{\heartsuit} \simeq \text{Mod}(R)$. The corresponding truncation functors
are described below, up to quasi-isomorphism.
\begin{alignat*}{5}
				X:& \  ... \to  	X^{n-2}     && \to \ \ X^{n-1}              && \to  \ \ \ \  X^n                   && \to X^{n+1} && \to  ... \\
       \ctr{st}{n-1}X:& \ ... \to   X^{n-2} && \to  \text{ker}(d^{n-1}) && \to \ \ \ \ \ 0                     && \to \ \ \ 0    && \to ... \\
       \cotr{st}{n}X:& \ ... \to  \ \ 0           && \to \ \ \ \ 0                    && \to  \text{coker}(d^{n-1}) && \to X^{n+1} && \to ...
\end{alignat*}
\end{example}

\noindent In particular, notice that 
\[ \text{H}^i(\ctr{st}{n-1}X) =\begin{cases}
    \text{H}^i(X), & i \leqslant n-1,\\
    0, & i \geqslant n.
  \end{cases} \]
Similarly, 
\[ \text{H}^i(\cotr{st}{n}X) =\begin{cases}
    0, & i \leqslant n-1,\\
    \text{H}^i(X), & i \geqslant n.
  \end{cases} \]

\begin{example}[\textbf{Constant t-structures}]
Fix some Thomason subset $Z$. Consider the constant filtration $\Phi_{Z} = \{Z\}$. Invoking theorem \ref{classification}, the corresponding 
t-structure $(\mathscr{U}_{Z}, \mathscr{V}_{Z})$ 
is given by
\[ \mathscr{U}_{Z} = \{X \in \D(R) \ | \ \text{Supp}\text{H}^i(X) \subseteq Z, \ i \in \mathbb{Z} \}. \]
It is immediate from theorem \ref{classification}, that these are exactly those t-structures with the property that both the aisle and coaisle are closed under shifting in both direction.
Furthermore, orthogonality gives us that  $\text{D}(R)_{Z}^{\heartsuit} \simeq 0$.

In this case, we have that both $\mathscr{U}_{Z}$ and $\mathscr{V}_{Z}$ are themselves thick, triangulated subcategories of $\D(R)$. 
In fact, the categories $\mathscr{U}_{Z}$ are exactly the 
smashing subcategories of $\D(R)$, see theorem 3.3 \cite{neem1}.
Furthermore, the connective truncation functor is the local cohomology functor $\textbf{R}\Gamma_{Z}$. See appendix A for more details on this functor.
\end{example}

\begin{example}[\textbf{Tilting t-structures}]
Fix some Thomason subset $Z$. Consider the t-structure corresponding to the filtration
\[ \Psi = \begin{cases}
    \Spec, & i \leqslant 0,\\
    Z, & i = 1, \\
    \emptyset, & i \geqslant 2.
  \end{cases} \]
The truncation functors
are described below
\begin{alignat*}{5}
				X:& \  ... \to  	X^{n-2}     && \to \ \ X^{n-1}              && \to  \ \ \  X^n                   && \to X^{n+1} && \to  ... \\
       \ctr{\Psi}{n-1}X:& \ ... \to   X^{n-2} && \to  \ \  X^{n-1}    && \to \ \ \  M                     && \to \ \ \ 0    && \to ... \\
       \cotr{\Psi}{n}X:& \ ... \to  \ \ 0           && \to \ \ \ \ 0                    && \to \  X^n/M  && \to X^{n+1} && \to ...
\end{alignat*}
\noindent with $M$ given by the pullback square
\begin{center}
\begin{tikzpicture}[>=stealth]
   \matrix (m) [
       matrix of math nodes,
       row sep=2em,
       column sep=1em,
       text height=1.5ex,
       text depth=0.25ex
     ] {                                        M        &              \text{ker}(d^{n})             \\

                   \Gamma_Z(\text{H}^n(X))     &            \text{H}^n(X).     \\
     };
     \path[->]            (m-1-1) edge    (m-1-2)
					(m-1-2) edge  (m-2-2)
					
 					(m-1-1) edge    (m-2-1)
					(m-2-1) edge  (m-2-2);
\end{tikzpicture}
\end{center}
\noindent This t-structure has been known about for some time, and is constructed from the standard t-structure by a process called tilting, which creates a new t-structure
out of a torsion pair on the heart of another. In \cite[1.1.2]{polishchuk2}, Polishchuk shows that two t-structures $\Phi$ and $\Psi$ are a tilt away from each when
\[ \ais{\Phi}{0}{\U} \subseteq \ais{\Psi}{0}{\U} \subseteq \ais{\Phi}{1}{\U}. \]
For corresponding Thomason filtrations $\Phi = \{Z^i\}$ and $\Psi = \{W^i\}$, this amounts to asking that $Z^i \subseteq W^i \subseteq Z^{i-1}$ for all $i$. 
\end{example}

The following result is also from \cite[3.11]{atjls} and characterises the coaisle corresponding to a Thomason filtration. 

\begin{prop}
\label{coaisle}
Let $R$ be a Noetherian ring, let $\Phi = \{ Z^i \}$ be a Thomason filtration on $\emph{Spec}R$ corresponding to an aisle $\mathscr{U}_\Phi$ in $\emph{D}(R)$. Then the corresponding coaisle 
is given by
\[ \mathscr{V}_\Phi    = \{ Y \in \emph{D}(R) \ | \ \ctr{st}{i}\emph{\textbf{R}}\Gamma_{Z^{i+1}}Y \simeq 0, \ i \in \mathbb{Z} \} \]
\end{prop}

\vspace{0.1cm}

\section{t-structures on the bounded derived category}

We will now discuss t-structures on the bounded derived category of a Noetherian ring. 
The following is a special case of corollary 3.12 in \cite{atjls}.

\begin{prop}
Let $R$ be a Noetherian ring, and let $(\mathscr{E}, \mathscr{F})$ be a t-structure on $\emph{D}^b(R)$. 
Then it extends to a compactly generated t-structure $(\mathscr{U}, \mathscr{V})$ on $\emph{D}(R)$.
\end{prop}

As a consequence of this and theorem \ref{classification}, we may naturally assign a unique Thomason filtration to any t-structure on $\Db(R)$.
We now introduce a condition that is necessarily satisfied by any Thomason filtration corresponding to a t-structure on $\Db(R)$, and see that for a large class of rings
 this condition is sufficient for a such t-structure to 
exist. A Thomason filtration $\{Z^i\}$ on $\Spec$ is said to satisfy the \textbf{weak cousin condition} if 
\[ \parbox[c]{3in}{\centering \text{for all} $i$, \text{for all} $\mathfrak{p} \in Z^i$, \text{and for all primes} 
$\mathfrak{q} \subsetneq \mathfrak{p}$ \text{maximal under} $\mathfrak{p}$, \text{we have that} $\mathfrak{q} \in Z^{i-1}$.}  \]
We will call such filtrations weak cousin Thomason
filtrations. Observe that the filtrations corresponding to the standard t-structure and the trivial t-structures both clearly satisfy the weak cousin condition, but the
filtration corresponding to the constant t-structure does not. The following is from \cite[4.5]{atjls}.

\begin{prop}
\label{weakcousin}
Let $R$ be a Noetherian ring, and let $\Phi = \{Z^i\}$ be a Thomason filtration on $\emph{Spec}R$ corresponding to a t-structure $(\mathscr{U}_\Phi, \mathscr{V}_\Phi)$ 
on $\emph{D}^{b}(R)$.
Then $\Phi$ satisfies the weak cousin condition.
\end{prop}

We will now recall basic results on dualizing complexes. A complex $X \in \D(R)$ is said to be \textbf{reflexive} with respect to $D \in \D(R)$ if the morphism
\[ X \to \textbf{R}\text{Hom}_R( \textbf{R}\text{Hom}_R(X, D), D) \]
is an equivalence. Let $D \in \Db(R)$ be quasi-isomorphic to a bounded complex of injective modules, then the following are equivalent \cite[2.1]{harts}

\begin{enumerate}
\item The contravariant functor $\textbf{R}\text{Hom}_R(-, D): \Db(R)^{\textit{op}} \to \Db(R)$ is a triangulated duality quasi-inverse to itself.

\item Every finitely generated $R$-module is reflexive with respect to $D$.

\item The complex $R[0]$ is reflexive with respect to $D$.
\end{enumerate}
Such a complex $D$ is said to be a \textbf{dualizing complex} for $R$. The existence of a dualizing complex is satisfied by all rings of finite Krull dimension that are quotients of a Gorenstein ring,
see corollary 1.4 in \cite{kawa}.
In particular any regular ring of finite Krull dimension admits a dualizing complex. The following theorem is due to  Alonso Tarrio, Jeremias Lopez and Saorin \cite[6.9]{atjls}.

\begin{thm}
\label{bclassification}
Let $R$ be a Noetherian ring that admits a dualizing complex, and let $\Phi = \{Z^i\}$ be a Thomason filtration on $\emph{Spec}R$. Then
 $\mathscr{U}_\Phi$ restricts to a 
t-structure on $\emph{D}^{b}(R)$ if and only if $\Phi$ satisfies the weak cousin condition.
\end{thm}

We will now proceed to characterise those filtrations corresponding to the bounded t-structures on $\Db(R)$. First we need the observation that t-structures 
on this category will behave well with respect to ring products.

\begin{lem}
\label{tstrringprod}
Let $R, S$ be rings. We have a bijection
\begin{align*}
            \left\{ \parbox[c]{0.9in}{\centering \small
                       \emph{t-structures on} $\emph{D}^b(R\times S)$}
            \right\}
            \longleftrightarrow
            \left\{ \parbox[c]{0.9in}{\centering \small
                        \emph{t-structures on} $\emph{D}^b(R)$}
            \right\} \times
            \left\{ \parbox[c]{0.9in}{\centering \small
                        \emph{t-structures on} $\emph{D}^b(S)$}
            \right\}
\end{align*}
\end{lem}
\begin{proof}
This follows from the fact that $\Db(R \times S) \simeq \Db(R) \oplus \Db(S)$.
\end{proof}

\begin{prop}
\label{boundedness}
Let $R$ be a Noetherian ring of finite Krull dimension, let $(\U, \V)$ be a t-structure on $\emph{D}^b(R)$, let $\Phi_{\U} = \{Z^i\}$ be the corresponding Thomason filtration
in $\emph{Spec}R$, and let $\emph{Spec}R = \bigsqcup_{i = 1}^n \emph{Spec}R_i$ be the decomposition of connected components. Then the following are equivalent:
\begin{enumerate}
\item The t-structure $(\U, \V)$ is bounded.

\item For all $i$, the corresponding t-structure on $\emph{D}^b(R_i)$ is non-trivial.

\item The Thomason filtration $\Phi_{\U} = \{Z^i\}$, terminates above at $\emph{Spec}R$, and below at $\emptyset$.
\end{enumerate}
\end{prop}

\begin{proof} 
\quad 

(1)$\Rightarrow$(2): Consider the t-structure $(\U, \V)$ projected onto each $\Db(R_i)$, as in
lemma \ref{tstrringprod}. Towards a contradiction, suppose there is some $i$ such that this t-structure
is trivial, such that $\ctr{}{n} \simeq 0$ or $\cotr{}{n} \simeq 0$. Take some $0 \not \simeq X \in \Db(R_i)$ viewed as an object of $\Db(R)$.
Since our t-structure is bounded it follows that there exists $n$ such that $\ctr{}{n}\cotr{}{-n}X \simeq X$. This gives us a contradiction.

(2)$\Rightarrow$(3): For each $i$, consider the corresponding t-structure on $\Db(R_i)$ along with its corresponding Thomason filtration on $\text{Spec}R_i$.
Since each corresponding t-structure on $\Db(R_i)$ is non-trivial, and $R$ has finite Krull dimension, corollary 4.8.3 in \cite{atjls} tells us that each of the corresponding filtrations terminates above
at $\text{Spec}R_i$ and below at $\emptyset$ respectively. Therefore, $\Phi_{\mathscr{U}}$ must also terminate above at $\Spec$ and below at $\emptyset$.

(3)$\Rightarrow$(1): Let $Z^r = \Spec$. Then by considering theorem \ref{classification}, we see that $\ais{}{r}{\Db(R)} \subseteq \mathscr{U}$. Let $Z^s = \emptyset$, then by considering
proposition \ref{coaisle}, we see that $\cais{}{s}{\Db(R)} \subseteq \mathscr{V}$. Since objects in $\Db(R)$ are cohomologically bounded, each $X \in \Db(R)$ is contained in some 
$\ais{}{n}{\Db(R)} \cap \cais{}{-n}{\Db(R)}$. Therefore each $X$ is contained in some $\ais{}{n}{\U} \cap \cais{}{-n}{\V}$, and $(\U, \V)$ is bounded.
\end{proof}

\vspace{0.1cm}

\section{Results on truncations}

Before we prove our main results in the next section, we will need to make some observations about the connective truncation functors for t-structures on $\D(R)$, with reference
to their corresponding Thomason filtrations. This will allow us to use the infinite generation of certain local cohomology groups to contradict the restriction of some t-structures to $\Perf(R)$.

\begin{lem}
\label{lem6}
Let $R$ be a Noetherian ring, let $(\U, \V)$ be a t-structure on $\emph{D}(R)$,
let $\Phi_{\U} = \{ Z^i \}$ be the corresponding Thomason filtration on 
$\emph{Spec}R$, and let $M$ be an $R$-module with $\emph{dim}M = d$, considered as a complex in 
$\emph{D}(R)$ in degree $0$.
\begin{enumerate}
\item If $Z^i = Z$ for $i \geqslant 0$ then $\ctr{\Phi}{0}M \simeq \emph{\textbf{R}}\Gamma_ZM$.

\item If $Z^i = Z$ for $i \leqslant d$ then $\ctr{\Phi}{0}M \simeq \emph{\textbf{R}}\Gamma_ZM$.
\end{enumerate}
\end{lem}

\begin{proof}
\begin{enumerate}

\item Since $Z \subseteq Z^i$ for all $i$, it is clear that $\U_Z \subseteq
\U$. Since, $\textbf{R}\Gamma_Z$ and $\ctr{\Phi}{0}$ are both the identity on their image, it follows that 
$\textbf{R}\Gamma_Z \circ \ctr{\Phi}{0} \simeq
\ctr{\Phi}{0} \circ \textbf{R}\Gamma_Z \simeq \textbf{R}\Gamma_Z$.
Set $\cotr{Z}{1}M = \text{cofib}(\textbf{R}\Gamma_ZM \to M)$. 
We know that  $\textbf{R}\Gamma_ZM \in \U$, and we know that the triangles corresponding to a t-structures are necessarily
unique, up to equivalence. Therefore, in order to see that the triangles
\[ \textbf{R}\Gamma_ZM \to M \to \cotr{Z}{1}M \]
and
\[ \ctr{\Phi}{0}M \to M \to \cotr{\Phi}{1}M \]
agree, it will suffice to show that $\cotr{Z}{1}M \in \cais{}{1}{\V}$. 
If we consider a shifted version of proposition \ref{coaisle} this amounts to showing that, for all $i$
\[ \ctr{st}{i}\text{\textbf{R}}\Gamma_{Z^{i}}\cotr{Z}{1}M \simeq 0. \]
When $i \geqslant 0$, then 
\[ \text{\textbf{R}}\Gamma_{Z^{i}}\cotr{Z}{1}M \simeq \text{\textbf{R}}\Gamma_{Z}\cotr{Z}{1}M \simeq 0. \]
To see the case when $i \leqslant -1$, lemma 4.2 in \cite{atjls} give us that 
$\textbf{R}\Gamma_{Z^{i}}\cotr{Z}{1}M \in \cais{}{0}{\D(R)}$, so we are done.

\item Similarly, $Z^i \subseteq Z$ for all $i$, so it is clear that $\U \subseteq
\U_Z$, and by orthogonality $\V_Z \subseteq
\V$. Since $\cotr{Z}{1}M \in \cais{}{1}{\V}$,
if we show that $\textbf{R}\Gamma_ZM \in \U$, then this will give us an equivalence of triangles, as above. 
By \cite[3.5.7(a)]{bruns} we know that 
\[\text{H}^i(\textbf{R}\Gamma_ZM) = \text{H}_Z^i(M) = 0\] for $i > d$.
Since, $Z^i = Z$ for $i \leqslant d$, it follows that 
\[ \text{Supp} \text{H}^i(\textbf{R}\Gamma_ZM) \subseteq Z^i \] for all $i$,
and therefore theorem \ref{classification} gives us that $\textbf{R}\Gamma_ZM \in \U$ and we are done.
\end{enumerate}
\end{proof}
For two Thomason filtrations $\Phi$ and $\Psi$, let the set theoretic notation $\Phi \cap \Psi$ denote its pointwise analogue on filtrations.

\begin{lem}
\label{lem7}
Let $R$ be a Noetherian ring, let $\Phi_{st}$ be the filtration corresponding to the standard t-structure $(\U_{st}, \V_{st})$ on $\emph{D}(R)$,
 and let $\Psi = \{Z_i \}$ be some Thomason filtration with corresponding t-structure $(\U_\Psi, \V_\Psi)$ on $\emph{D}(R)$. 
Then the Thomason filtration $\Phi_{st} \cap \Psi$ has corresponding aisle $\U_{st} \cap \U_\Psi$, and corresponding connective truncation functor 
$\ctr{st}{0} \circ \ctr{\Phi}{0}$.
\end{lem}

\begin{proof}
It is clear from theorem 
 \ref{classification}, that the t-structure corresponding to the filtration $\Phi_{st} \cap \Psi$ has aisle $\U_{st} \cap \U_\Psi$. 

Note that $\ctr{st}{0}$ either preserves homology or takes it 
to zero. It follows that  
$\ctr{st}{0} (\U_\Psi) 
\subseteq \U_\Psi$. Therefore, 
$\text{im}(\ctr{st}{0} \circ \ctr{\Psi}{0}) \subseteq 
\U_{st} \cap \U_\Psi$. But both 
$\ctr{st}{0}$ and $\ctr{\Psi}{0}$ are 
the identity on $\U_{st} \cap \U_\Psi$, so it follows that 
$\text{im}(\ctr{st}{0} \circ 
\ctr{\Psi}{0}) = \U_{st} \cap \U_\Psi$. Since, $\ctr{st}{0}$ and $\ctr{\Phi}{0}$ are both right adjoint to the inclusions 
of their respective images, it follows that $\ctr{st}{0} \circ 
\ctr{\Phi}{0}$ is right adjoint to the inclusion of its image $\U_{st} \cap \U_\Psi$, and so must be the corresponding truncation functor.
\end{proof}

\begin{prop}
\label{summand}
Let $R$ be a Noetherian ring, 
let $[a, b] \subseteq \mathbb{Z}$ be an interval
with $a \leqslant b$, let $\Psi = \{ Z^i \}$ be a Thomason filtration such that 
$\emph{height}(Z^i) = h$ for all $i \in [a, b]$, let $(\U_\Psi, \V_\Psi)$ be the corresponding t-structures on $\D(R)$, let $\mathfrak{p}$ be a minimal prime of
$Z_b$, and let $M$ be an $R$-module. Then for all $i \in [a, b]$,
$\emph{H}^i(\ctr{\Psi}{0}M[-a])_\mathfrak{p}$ contains $\emph{H}_{\mathfrak{p}R_\mathfrak{p}}^{i-a}(M_\mathfrak{p})$
as a direct summand.
\end{prop}

\begin{proof}
Without loss of generality we may set $a=0$. 
Define two new Thomason filtrations by
$\Omega \coloneqq \{ Z^b \ i \leqslant b \ , \ Z^i \ i > b \}$ and  
$\Sigma \coloneqq \{ Z^i \ i < 0 \ , \ Z^0 \ i \geqslant 0 \}$.
 Given that $\Omega$ and the constant filtration $\{ Z^b \}$ agree for $i \leqslant b$, we may intersect them both with the filtration
 $\{ \Spec \ i \leqslant b, \ \emptyset \ i > b \}$ 
(the standard filtration shifted), and then
 applying lemma \ref{lem7} gives us that 
 $\ctr{st}{b} \circ \ctr{\Omega}{0} \simeq \ctr{st}{b} \circ \textbf{R}\Gamma_{Z^{b}}$.
By lemma \ref{lem6} (1), $\ctr{\Sigma}{0}M \simeq \textbf{R}\Gamma_{Z^{0}}M$ so $\ctr{st}{b}\ctr{\Sigma}{0}M \simeq \ctr{st}{b}\textbf{R}\Gamma_{Z^{0}}M$.
By construction, $\U_\Omega \subseteq \U_\Psi \subseteq \U_\Sigma$, so we have a pair of natural maps
\[ \ctr{\Omega}{0}M
 \to \ctr{\Psi}{0}M \to 
 \ctr{\Sigma}{0}M .\]
Applying $\ctr{st}{b}$ we get 
\[ \ctr{st}{b}\ctr{Z^b}{0}M \simeq \ctr{st}{b} \ctr{\Omega}{0}M
 \to \ctr{st}{b} \ctr{\Psi}{0}M \to 
 \ctr{st}{b} \ctr{\Sigma}{0}M \simeq \ctr{st}{b} \ctr{Z^0}{0}M.\]
Take some $i \in [0, b]$ and some prime
$\mathfrak{p}$ minimal in $Z^b$. Since 
$\text{height}(Z^b) = \text{height}(Z^0)$
it follows that $\mathfrak{p}$ is also minimal in $Z^0$.
So
applying $\text{H}^i(-)_\mathfrak{p}$ and using lemma \ref{lem1} we get
\[\text{H}_{\mathfrak{p}R_{\mathfrak{p}}}^{i}(M_\mathfrak{p}) \to 
\text{H}^{i}(\ctr{\Psi}{0}M)_\mathfrak{p} \to 
\text{H}_{\mathfrak{p}R_{\mathfrak{p}}}^{i}(M_\mathfrak{p}) \]
such that the composition is the identity. Therefore,
$\text{H}_{\mathfrak{p}R_{\mathfrak{p}}}^{i}(M_\mathfrak{p})$
is a summand of $\text{H}^{i}(\ctr{\Psi}{0}M)_\mathfrak{p}$.
\end{proof}

\vspace{0.1cm}

\section{t-structures on the category of perfect complexes}

We will now prove the main results of the paper. 
First note that by combining proposition \ref{exttstr}, theorem \ref{classification} and the fact that $\text{Ind}(\Perf(R)) \simeq \D(R)$, we see that any t-structure 
on $\Perf(R)$ may be extended to a compactly generated t-structure on $\D(R)$. Therefore, we may naturally assign a unique Thomason filtration to any t-structure on $\Perf(R)$. 
Furthermore, as a consequence 
of theorem \ref{bclassification} we have the following.

\begin{prop}
\label{regclassification}
Let $R$ be a regular, Noetherian ring of finite Krull dimension, and let $\Phi = \{Z^i\}$ be a Thomason filtration on $\emph{Spec}R$. Then
 $\mathscr{U}_\Phi$ restricts to a 
t-structure on $\emph{Perf}(R)$ if and only if $\Phi$ satisfies the weak cousin condition, and this t-structure is bounded if and only if $\Phi$ terminates 
above at $\emph{Spec}R$
and terminates below at $\emptyset$.
\end{prop}

\begin{proof}
As a consequence of the finite global dimension of regular rings (see \cite[4.4]{weib1}) we have that $\Db(R) \simeq \text{Perf}(R)$.
Any regular ring of finite Krull dimension admits a dualizing complex \cite[1.4]{kawa}.
Hence, this is a particular case of theorem \ref{bclassification}, and proposition \ref{boundedness} applies.
\end{proof}
The following lemma is essentially stating that proposition 2.9 in \cite{atjls} restricts to perfect complexes.
 \begin{lem}
\label{loctstr}
Let $R$ be a Noetherian ring, let $\mathfrak{p}\subset R$ be a prime ideal and let $\Phi = \{Z^i\}$ be a Thomason filtration corresponding to a t-structure $(\U_\Phi, \V_\Phi)$
 on $\emph{Perf}(R)$. Then the Thomason filtration $\Psi = \{Z^i \cap \emph{Spec}R_{\mathfrak{p}}\}$ on $\emph{Spec}R_\mathfrak{p}$, gives a t-structure 
$(\mathscr{E}_\Psi, \mathscr{V}_\Psi)$ 
on $\emph{Perf}(R_{\mathfrak{p}})$.
\end{lem}

\begin{proof}
Using proposition \ref{exttstr}, extend $(\U_\Phi, \V_\Phi)$ to a t-structure $(\U_\Phi', \V_\Phi')$ on $\text{D}(R)$. Using proposition 2.9 in \cite{atjls}, 
the  Thomason filtration $\Psi = \{Z^i \cap \text{Spec}R_{\mathfrak{p}}\}$ gives a t-structure $(\mathscr{E}_\Psi', \mathscr{V}_\Psi')$ on $\D(R_{\mathfrak{p}})$,
such that the functor $\D(R) \to \D(R_\mathfrak{p})$ is t-exact.

Set $\mathscr{E}_\Psi = \mathscr{E}_\Psi'
 \cap \Perf(R_{\mathfrak{p}})$ and 
$\mathscr{F}_\Psi = \mathscr{F}_\Psi'
 \cap \Perf(R_{\mathfrak{p}})$. We claim these are the aisle and coaisle of a
 t-structure on $\Perf(R)$.
Conditions $(1)$ and $(2)$ of definition \ref{tstr} are clearly satisfied. To see condition (3), it suffices to show that for   all $X \in \Perf(R_{\mathfrak{p}})$ and for all $n$,
 we have $\cotr{\Psi}{n}X \in \Perf(R_{\mathfrak{p}})$.

Since $\Phi$ gives a t-structure on $\Perf(R)$, we have that $\cotr{\Phi}{n}R \in \Perf(R)$. 
Given that ring localizations take perfect complexes to perfect complexes, and that $\D(R) \to \D(R_\mathfrak{p})$ is t-exact, it follows that 
\[ (\cotr{\Phi}{n}R)_{\mathfrak{p}} \simeq \cotr{\Psi}{n}R_{\mathfrak{p}} \in \Perf(R_{\mathfrak{p}}). \]
Every object $X \in \Perf(R_{\mathfrak{p}})$ can be constructed from $R_\mathfrak{p}$ with finite direct sums, mapping cones and retracts. 
The functor $\cotr{\Psi}{n}$ is a left adjoint and so preserves
colimits. Therefore $\cotr{\Psi}{n}X $ can be constructed from $\cotr{\Psi}{n}R_\mathfrak{p}$ with finite direct sums, mapping cones and retracts, and must be contained in $\Perf(R_{\mathfrak{p}})$.
\end{proof}

\begin{lem}
\label{doesnotrestrict}
Let $R$ be a singular, Noetherian ring, let $\Phi = \{Z^i\}$ be a Thomason filtration on $\emph{Spec}R$, such that there is some $r$, with $Z^r = \emptyset$, and
some $s$ with $Z^s$ containing a maximal ideal $\mathfrak{m}$ corresponding to a singular point. Then the t-structure $(\U_\Phi, \V_\Phi)$ on $\emph{D}(R)$ does not restrict to 
$\emph{Perf}(R)$.
\end{lem}

\begin{proof}
Towards a contradiction, suppose that $(\U_\Phi, \V_\Phi)$ restricts to $\Perf(R)$. Let $W^i = Z^i \cap \text{Spec}R_{\mathfrak{m}}$.
Using lemma \ref{loctstr}, the filtration $\Psi = \{W^i\}$ gives a t-structure $(\U_\Psi, \V_\Psi)$ on $\D(R_\mathfrak{m})$, which restricts to a t-structure 
$(\mathscr{E}_\Psi, \mathscr{F}_\Psi)$ on $\Perf(R_{\mathfrak{m}})$.
Since $Z^r = \emptyset$, we have that $W^r = \emptyset$, and since $\mathfrak{m} \in Z^s$, we have that $\mathfrak{m} \in Z^s$.
This reduces to the local case, so for the remainder of the proof set $R = R_{\mathfrak{m}}$, and let $R/\mathfrak{m} = k$ be the corresponding residue field.

Without loss of generality, we may assume that the highest $i$ such that $W^i$ is non-empty is $i=0$. Let $( \mathscr{U}_{st}, \mathscr{V}_{st})$ denote the standard t-structure on $\D(R)$.
Since $W^i = 0$ for $i > 0$, theorem \ref{classification} tells us that $\mathscr{E}_\Psi \subseteq \mathscr{U}_{st}$.
Since the truncation functors are the identity on their image we see that $\ctr{\Psi}{-1} \circ \ctr{st}{-1}
\simeq \ctr{\Psi}{-1}$. 

Since $W^0$ is closed under specialization and non-empty, and $\mathfrak{m}$ is the only maximal ideal in $R$, it is immediate that
$\mathfrak{m} \in W^0$. 
Let $\mathfrak{m} = (m_1, ... , m_n)$ be a set of generators. Consider the Koszul complex $K(\overbar{m}) \in \Perf(R) \subseteq \D(R)$. 
By assumption, $\Psi$ gives a t-structure on $\Perf(R)$ and for all $n$, we have that $\cotr{\Psi}{n} K(\overbar{m}) \in \Perf(R)$.

Using \cite[1.6.5(b)]{bruns}, we see that
\[ \text{Supp}\text{H}^i(K(\overbar{m})) \subseteq V(\mathfrak{m}) \subseteq W^i \]
for all $i \leqslant 0$.  It follows that $\ctr{st}{-1}K(\overbar{m}) \in 
\U_\Psi$. 
Therefore, $\ctr{\Psi}{-1} \ctr{st}{-1} K(\overbar{m}) \simeq \ctr{st}{-1} K(\overbar{m})$. 
Therefore
\[ \ctr{\Psi}{-1} K(\overbar{m}) \simeq \ctr{\Psi}{-1} \ctr{st}{-1} K(\overbar{m}) \simeq \ctr{st}{-1} K(\overbar{m})\]
 and  $\cotr{\Psi}{0} K(\overbar{m}) \simeq \cotr{st}{0} K(\overbar{m})$. 
Given that $ K(\overbar{m})$ has vanishing 
homology in positive degrees, we have that 
\[\cotr{\Psi}{0} K(\overbar{m}) \simeq \cotr{st}{0} K(\overbar{m}) \simeq \text{H}^0(K(\overbar{m})) \simeq k. \]
Since $k$ is the residue field corresponding to a singular point, we know that 
 $\text{proj.dim}k \\ = \infty$ \cite[4.4.16]{weib1}. It follows that $\cotr{\Psi}{0} K(\overbar{m}) \simeq k$ cannot be a perfect complex, giving us our contradiction.
\end{proof}

Now we will prove theorems $1.2$ and $1.3$.

\begin{thm}
\label{nobounded}
Let $R$ be a singular, Noetherian ring of finite Krull dimension. Then $\emph{Perf}(R)$ admits no bounded t-structure.
\end{thm}

\begin{proof}
Towards a contradiction, suppose that $\Phi = \{Z^i\}$ is a Thomason filtration on $\Spec$ giving rise to a bounded t-structure
$(\U_\Phi, \V_\Phi)$ on $\Perf(R)$.
Since this is a bounded t-structure, there must be some $n$ such that 
$\ctr{\Phi}{n}R \simeq R$. So considering theorem \ref{classification}, we see that the same $n$ must also have that $Z^n = \Spec$. 

We claim that there can only be a finite string of non-empty $Z^i$'s of the same non-zero height.
Towards a contradiction, suppose that for all $i$
in some interval $[a, a+d]$, we have $Z^i$ is a non-empty and that 
$\text{height}(Z^i) = h \geqslant 1$ . Take some prime $\mathfrak{p}$, minimal in $Z^{a+d}$.
Proposition \ref{summand} tells us that 
$\text{H}^{h+a}(\ctr{\Phi}{0}R[-a])_{\mathfrak{p}}$ contains $\text{H}_{\mathfrak{p}R_{\mathfrak{p}}}^h(R_{\mathfrak{p}})$ as a direct summand, and is therefore infinitely generated over $R$
by corollary \ref{infg}. It follows that $\ctr{\Phi}{0}R[-a]$ cannot be a perfect complex, and $(\U, \V)$ cannot give a t-structure on $\Perf(R)$. 

Since $R$ is Noetherian and has finitely many irreducible components, it follows from the above claim that $\Phi$ must terminate below at $\emptyset$ or a Thomason subset of height $0$.
If $\Phi$ terminated below at some non-empty Thomason subset $Z^n$, then for all $I=(\overbar{x})$, with $V(I) \subset Z^n$,
we would have that $\ctr{\Phi}{n}K(\overbar{x}) \simeq K(\overbar{x})$ for all $n$, contradicting boundedness. Therefore, $\Phi$ must terminate below at $\emptyset$
and above at $\Spec$. Lemma \ref{doesnotrestrict} shows that this t-structure cannot exist on $\Perf(R)$, giving us our contradiction.
\end{proof}

\begin{thm}
\label{singconn}
Let $R$ be a singular, irreducible, Noetherian ring of finite Krull dimension $d$, and let $(\U, \V)$ be a t-structure on $\emph{Perf}(R)$. 
Then either $\U = 0$ or $\U = \emph{Perf}(R)$.
\end{thm}

\begin{proof}
In the case that $d=0$, then $\Spec$ contains a single point. The only possible filtrations are the standard one up to shifting, and the two trivial ones. Lemma \ref{doesnotrestrict}
contradicts the standard filtration, so we are done.

Let $d \geqslant 1$. Let $\Phi_\U = \{Z^i\}$ be the Thomason filtration on $\Spec$ corresponding to $(\U, \V)$. 
Similarly to the proof of theorem \ref{nobounded}, observe that 
there can only be a finite string of $Z^i$'s of the same non-zero height as otherwise $\ctr{\Phi}{0}R[-a]$ would have infinitely generated homology for some $a$.

Since $\Spec$ is irreducible, the only Thomason subset of height $0$ is $\Spec$. Since $R$ has finite Krull dimension and we know from above that there can be 
only be a finite string of $Z^i$'s of the same non-zero height, we  see
 that $\Phi$ must terminate below at $\emptyset$ or $\Spec$, and $\Phi$ must terminate above at $\emptyset$ or $\Spec$. Lemma \ref{doesnotrestrict}
contradicts the possibility that $\Phi$ terminates below at $\emptyset$ and above at $\Spec$. Therefore, $\Phi$ must be either the constant filtration
at $\emptyset$ or the constant filtration at $\Spec$.
\end{proof}

\vspace{0.1cm}

\appendix

\section{Local cohomology}

Let $R$ be a commutative, Noetherian ring. Let $Z \subseteq \text{Spec}R$ be a Thomason subset and $M$ a $R$-module. Let $E(R/\mathfrak{p})$ be the indecomposable 
injective $R$-module corresponding to some prime ideal $\mathfrak{p} \in \Spec$ (see appendix A in \cite{24hours}).
Define a submodule 
$\Gamma_ZM \subseteq M$ by the exact sequence
\[ 0 \to \Gamma_ZM \to M \to \prod_{\mathfrak{p} \notin Z} M_{\mathfrak{p}}. \]
We have the following facts about $\Gamma_Z$ (see \cite[section 9]{bens}):
\begin{enumerate}
\item The assignment $M \mapsto \Gamma_ZM$ is a left exact, additive functor on $\text{Mod}(R)$

\item For some ideal $I \subset R$, denote $\Gamma_I \coloneqq \Gamma_{V(I)}$, and we have
\[ \Gamma_IM = \{ x \in M \ | \ I^nx = 0 \ \text{for some} \ n \in \mathbb{Z} \}. \]

\item For arbitrary Thomason subset $Z$
\[ \Gamma_ZM = \bigcup_{V(I) \subseteq Z} \Gamma_{I}M. \]

\item For each $\mathfrak{p} \in \text{Spec}R$
\[ \Gamma_Z(E(R/\mathfrak{p})) =\begin{cases}
    E(R/\mathfrak{p}), & \mathfrak{p} \in Z.\\
    0, & \text{otherwise}.
  \end{cases} \]
\end{enumerate}
Let $\textbf{R}\Gamma_Z: \text{D}(R) \to \text{D}(R)$ be the right derived
functor of $\Gamma_Z$. For each $X \in \text{D}(R)$, we define the \textbf{local cohomology} modules 
of $X$ with respect to $Z$ to be 
\[ \text{H}_Z^i(X) \coloneqq \text{H}^{i}(\textbf{R}\Gamma_ZX). \]
When $Z = V(I)$ just write 
$\text{H}_I^i(X)$ for $\text{H}_Z^i(X)$.
When $(R,\mathfrak{m})$ is a Noetherian, local ring, and $M$ a finitely generated $R$-module we 
have the following well known facts:
\begin{enumerate}

\item \label{1} $\text{H}_\mathfrak{m}^i(M)$ is Artinian for $i \geqslant 0$ \cite[3.5.4(a)]{bruns}.

\item \label{2} When $\text{dim}M \geqslant 1$, set $N = M/\text{H}_\mathfrak{m}^0(M)$. Then 
$\text{dim}M = \text{dim}N$, $\text{depth}N \geqslant 1$, and 
$\text{H}_\mathfrak{m}^i(N) = \text{H}_\mathfrak{m}^i(M)$ for $i \geqslant 1$ \cite[2.1.7]{brod}.

\item $\text{H}_\mathfrak{m}^i(M)$ can only be non-zero in the range
 $\text{depth}M \leqslant i \leqslant \text{dim}M$ \cite[3.5.7(a)]{bruns}.

\end{enumerate}
For $x \in R$ let $\mu_M^x : M \to M$ denote multiplication by $x$. 

\begin{lem}
\label{lem5}
Let $(R,\mathfrak{m})$ be a Noetherian, local ring and let M be a finitely generated $R$-module.
Then there exists $x \in \mathfrak{m}$ such that $\emph{ker}\mu_M^x$ has finite length.
\end{lem}

\begin{proof} 
When $\text{dim}M = 0$ then $M$ is Artinian and any $x$ annhilating the whole module works. 
When $\text{dim}M \geqslant 1$ then set $N = M/\text{H}_\mathfrak{m}^0(M)$. By (\ref{2}) above, 
we can choose some $x \in \mathfrak{m}$ that is $N$-regular. It follows that 
$\text{ker}\mu_M^x \subseteq \text{H}_\mathfrak{m}^0(M)$ and must be Artinian, by (\ref{1}),
and therefore have finite length.
\end{proof} 

\noindent The following proof is based on \cite{put1} with some alterations.

\begin{prop}
\label{prop4}
Let $(R,\mathfrak{m},k,E)$ be a Noetherian, local ring, let $M$ a finitely generated non-zero $R$-module 
with $\emph{dim}M = r \geqslant 1$. Then $\emph{H}_\mathfrak{m}^r(M)$ is infinitely generated over $R$.
\end{prop}

\begin{proof}
We know from \cite[3.5.4(a)]{bruns} that $\text{H}_\mathfrak{m}^r(M)$ is Artinian so it suffices to show that 
$\ell(\text{H}_\mathfrak{m}^r(M)) = \infty$. We proceed by induction on $r$.

Let $r=1$. By \cite[2.1.7]{brod} we may assume that $\text{depth}(M) \geqslant 1$ so 
$\text{H}_\mathfrak{m}^0(M)=0$. Take $M$-regular $x \in \mathfrak{m}$. Since $\text{dim}M/xM=0$
we see that $\text{H}_\mathfrak{m}^0(M/xM)=M/xM$ and $\text{H}_\mathfrak{m}^1(M/xM)=0$. Therefore, 
applying $\text{H}_\mathfrak{m}^*(-)$ to the short exact sequence
\[ 0 \to M \xrightarrow[]{\mu_M^x} M \to M/xM \to 0 \]
gives us
\[ 0 \to M/xM \to \text{H}_\mathfrak{m}^1(M) \to \text{H}_\mathfrak{m}^1(M) \to 0 .\]
If $\ell(\text{H}_\mathfrak{m}^1(M)$ was finite, it would then follow that 
\[ \ell(M/xM)= \ell(\text{H}_\mathfrak{m}^1(M)) - \ell(\text{H}_\mathfrak{m}^1(M)) = 0 \]  
and that $M/xM=0$. By Nakayama's lemma, we would then get that $M=0$, which is a contradiction. Therefore, 
$\ell(\text{H}_\mathfrak{m}^1(M))=\infty$ and $\text{H}_\mathfrak{m}^1(M)$ is infinitely generated.

Now assume the result for $r=s$, and let $M$ have $r=s+1$. Since $\text{H}_\mathfrak{m}^i(M)$
is Artinian, Matlis duality gives us that $\text{H}_\mathfrak{m}^i(M)^\vee = \text{Hom}(\text{H}_\mathfrak{m}^i(M), E)$ is
finitely generated.

By lemma \ref{lem5} we get $x \in \mathfrak{m}$ such that 
$\text{ker}\mu_{\text{H}_\mathfrak{m}^i(M)^\vee}^x$ has finite length. Applying
 $\text{Hom}(-,E)$ and invoking Matlis duality we get that
$C=\text{coker}\mu_{\text{H}_\mathfrak{m}^i(M)}^x$ also has finite length. 

Again, by \cite[2.1.7]{brod} we may assume that $\text{depth}(M) \geqslant 1$, so it follows that $x$ is 
necessarily $M$-regular. Note, $\text{dim}M/xM=s$ so $\text{H}_\mathfrak{m}^{s+1}(M/xM)=0$
and by the induction hypothesis $\ell(\text{H}_\mathfrak{m}^s(M/xM))=\infty$. Again,
applying $\text{H}_\mathfrak{m}^*(-)$ to the short exact sequence
\[ 0 \to M \xrightarrow[]{\mu_M^x} M \to M/xM \to 0 \]
gives us
\[ 0 \to C \to \text{H}_\mathfrak{m}^s(M/xM) \to \text{H}_\mathfrak{m}^{s+1}(M) \to \text{H}_\mathfrak{m}^{s+1}(M) \to 0 .\]
We have that $\ell(C)<\infty$ and $\ell(\text{H}_\mathfrak{m}^s(M/xM))=\infty$, so 
$\ell(\text{H}_\mathfrak{m}^{s+1}(M))=\infty$ and $\text{H}_\mathfrak{m}^{s+1}(M)$ is 
infinitely generated.
\end{proof}

\begin{cor}
Let $R$ be a Noetherian ring, let $\mathfrak{p} \in \emph{Spec}R$ with 
$\emph{height}(\mathfrak{p}) = h \geqslant 1$. Then $\emph{H}_{\mathfrak{p}}^h(R)$ is infinitely generated over $R$.
\end{cor}

\begin{proof}
Towards a contradiction, suppose that $\text{H}_{\mathfrak{p}}^h(R)$ was finitely generated as an 
$R$-module. Then $\text{H}_{\mathfrak{p}}^h(R)_{\mathfrak{p}}$ would be finitely generated as
an $R_{\mathfrak{p}}$-module. The Noetherian, local ring $(\mathfrak{p}R, \mathfrak{p})$ has $\text{dim}R_\mathfrak{p} = h$ and
\[ \text{H}_{\mathfrak{p}}^h(R)_{\mathfrak{p}} = 
\text{H}_{\mathfrak{p}R_{\mathfrak{p}}}^h(R_{\mathfrak{p}}) \]
which is infinitely generated by proposition \ref{prop4}. This gives us our contradiction.
\end{proof}

\begin{lem}
\label{lem1}
Let $R$ be a Noetherian ring, let $Z \subseteq \emph{Spec}R$ be a non-empty, Thomason
subset, let $\mathfrak{p}$ be a minimal prime of $Z$, and let $M$ an $R$-module. Then
\[ \emph{H}_Z^i(M)_\mathfrak{p} \simeq
 \emph{H}_{\mathfrak{p}R_\mathfrak{p}}^i(M_\mathfrak{p}). \]
\end{lem}

\begin{proof}
Let $M \to E_\bullet$ be an injective resolution of $A$-modules. 
Since $- \otimes R_\mathfrak{p}$ takes injective
$R$-mdodules to injective $R_\mathfrak{p}$-modules, 
$M_\mathfrak{p} \to E_\bullet \otimes R_\mathfrak{p}$
is an injective resolution of $R_\mathfrak{p}$-modules. The $R$-module $R_\mathfrak{p}$ is flat so 
\[ \text{H}_Z^i(M)_\mathfrak{p} = \text{H}^{i}(\Gamma_Z(E_\bullet))_\mathfrak{p} 
\simeq \text{H}^{i}(\Gamma_Z(E_\bullet)_\mathfrak{p}) \]
and 
\[  \text{H}_{\mathfrak{p}R_\mathfrak{p}}^i(M_\mathfrak{p}) = 
\text{H}^{i}(\Gamma_{\mathfrak{p}R_\mathfrak{p}}(E_\bullet \otimes R_\mathfrak{p})). \]
So it suffices to show that 
$\Gamma_Z(M)_\mathfrak{p} \simeq \Gamma_{\mathfrak{p}R_\mathfrak{p}}(M_\mathfrak{p})$.
Recall, $\Gamma_Z(M)$ is given by the exact sequence
\[ 0 \to \Gamma_Z(M) \to M \to \prod_{\mathfrak{q} \notin Z} M_{\mathfrak{q}}. \]
Applying $- \otimes R_\mathfrak{p}$ we get 
\[ 0 \to \Gamma_Z(M)_\mathfrak{p}  \to M_\mathfrak{p} \xrightarrow[]{f} 
\Big( \prod_{\mathfrak{q} \notin Z} M_{\mathfrak{q}} \Big)_\mathfrak{p}. \]
The condition that $\mathfrak{q} \notin V(\mathfrak{p}R_\mathfrak{p})$, is equivalent to the condition that
$\mathfrak{q} \subsetneq \mathfrak{p}$, and when this is satisfied then $(M_\mathfrak{p})_\mathfrak{q}
\simeq M_\mathfrak{q}$. Therefore we get that 
$\Gamma_{\mathfrak{p}R_\mathfrak{p}}(M_\mathfrak{p})$ is given by the exact sequence
\[ 0 \to \Gamma_{\mathfrak{p}R_\mathfrak{p}}(M_\mathfrak{p}) 
\to M_\mathfrak{p} \xrightarrow[]{g} \prod_{\mathfrak{q} \subset \mathfrak{p}} 
M_{\mathfrak{q}} . \]
The requirement that $\mathfrak{p}$ is a minimal prime of $Z$ implies that 
$\{\mathfrak{q} \ | \ \mathfrak{q} \subsetneq \mathfrak{p}\} \subseteq 
\{ \mathfrak{q} \ | \ \mathfrak{q} \notin Z\}$.
For some $\mathfrak{q}
\subsetneq \mathfrak{p}$,
any $x \in R \smallsetminus \mathfrak{p}$ acts invertibly on $M_{\mathfrak{q}}$. 
Therefore, any such $x$ also acts invertibly on $\prod_{\mathfrak{q} \subsetneq \mathfrak{p}} 
M_{\mathfrak{q}}$. Applying the universal property of ring localizations to the natural projection gives us a map 
\[ \Big( \prod_{\mathfrak{q} \notin Z} M_{\mathfrak{q}} \Big)_\mathfrak{p}
\to \prod_{\mathfrak{q} \subsetneq \mathfrak{p}} 
M_{\mathfrak{q}}. \]
This map factors $g$ through $f$ giving us
$\Gamma_Z(M)_\mathfrak{p} \subseteq \Gamma_{\mathfrak{p}R_\mathfrak{p}}(M_\mathfrak{p})$.
To see the other inclusion, elements of $\Gamma_{\mathfrak{p}R_\mathfrak{p}}(M_\mathfrak{p})$
are of the form $\frac{a}{r}$ with $a \in M$, $r \in R \smallsetminus \mathfrak{p}$, such that
$\mathfrak{p}^n \frac{a}{r} = 0$ for some $n$. 
This implies that $\mathfrak{p}^na = 0$, and since 
$\mathfrak{p} \in Z$, we see that $a \in \Gamma_Z(M)$. 
Therefore $\frac{a}{r} \in \Gamma_Z(M)_\mathfrak{p}$,
and 
$\Gamma_Z(M)_\mathfrak{p} \supseteq \Gamma_{\mathfrak{p}R_\mathfrak{p}}(M_\mathfrak{p})$.
\end{proof}

\begin{cor}
\label{infg}
Let $R$ be a Noetherian ring, and let $Z \subseteq \emph{Spec}R$ be a non-empty, Thomason
subset with $h = \emph{height}(Z) \geqslant 1$. 
Then $\emph{H}_{Z}^h(R)$ is infinitely generated over $R$.
\end{cor}

\begin{proof}
Towards a contradiction, suppose that $\text{H}_{Z}^h(R)$ was finitely generated as an 
$R$-module. Let $\mathfrak{p}$ be a minimal prime of $Z$. Then $\text{H}_{Z}^h(R)_{\mathfrak{p}}$ 
would be finitely generated as
an $R_{\mathfrak{p}}$-module. $(\mathfrak{p}R, \mathfrak{p})$ is a Noetherian, local ring
with $\text{dim}R_\mathfrak{p} = h \geqslant 1$ and by lemma \ref{lem1}
\[ \text{H}_{Z}^h(R)_{\mathfrak{p}} = 
\text{H}_{\mathfrak{p}R_{\mathfrak{p}}}^h(R_{\mathfrak{p}}) .\]
This is infinitely generated by proposition \ref{prop4}. This gives us our contradiction.
\end{proof}


\end{document}